\documentclass{article}
\usepackage{graphicx} 
\usepackage{amsfonts,amsmath,amssymb,amsthm}
\usepackage{tikz}
\usetikzlibrary{calc}

\usepackage[a4paper,margin=2cm]{geometry}

\theoremstyle{plain}
\newtheorem{theorem}{Theorem}[section]
\newtheorem{lemma}[theorem]{Lemma}
\newtheorem{proposition}[theorem]{Proposition}
\newtheorem{corollary}[theorem]{Corollary}

\theoremstyle{definition}
\newtheorem{definition}[theorem]{Definition}

\theoremstyle{remark}
\newtheorem{remark}[theorem]{Remark}

\title{Dynamics of the Longest-Edge Altitude Bisection Algorithm}
\author{Jérôme Michaud and Sergey Korotov}
\date{\today}

\begin{document}

\maketitle
\begin{center}
Department of Business and Mathematics, IEM, Mälardalen University, Västerås, Sweden \newline
emails: jerome.michaud@mdu.se,  sergey.korotov@mdu.se 
\end{center}

\begin{abstract}
We study a longest-edge based refinement scheme for triangulations, termed the longest-edge altitude bisection (LEAB), in which each triangle is subdivided by dropping the altitude from the vertex opposite to its longest edge. Using the normalized shape space of triangles introduced by Perdomo and Plaza in: Properties of triangulations obtained by the longest-edge bisection. \emph{Cent. Eur. J. Math.}, 12(12) (2014), 1796-1810, we show that LEAB admits a simple geometric description: the normalized left and right children of a triangle  in focus are obtained by intersecting the geodesic of right triangles with rays issued from the endpoints of the longest edge and explicit formulas for the mappings are derived. This characterization implies  an interesting observation that the associated refinement dynamics collapse the entire shape space onto the right-triangle geodesic in a single step and that every point on this geodesic is fixed. Two-sided bounds for the contraction of the mesh  size (discretization parameter) are derived. Also, applications and limitations of the method are briefly discussed.
\end{abstract}

\medskip

\noindent{\bf Kewords:} longest-edge altitude bisection, mesh refinement, refinement dynamics, normalization space, two-sided bounds

\medskip


\section{Introduction}
Longest-edge based refinement strategies play a central role in the adaptive refinement of triangular meshes. Among these, classical longest-edge bisection (LEB), in which a triangle is subdivided by joining the midpoint of its longest edge to the opposite vertex, has been extensively studied and is known to preserve shape regularity under repeated refinements \cite{RosenbergStenger1975,Rivara1984,Stynes1980}. Beyond its practical relevance, LEB has also attracted attention from a geometric and dynamical perspective. In particular, Perdomo and Plaza \cite{PerdomoPlaza2014} introduced a normalized shape space of triangles endowed with a hyperbolic metric and showed that LEB induces a non-expanding dynamical system with finitely many similarity classes.

In this work, we study a closely related but structurally simpler refinement strategy, which we call the \emph{longest-edge altitude bisection} (LEAB). Given a triangle, LEAB subdivides it by dropping the altitude from the vertex opposite to its longest edge,   thus producing two subtriangles. This construction is purely geometric and does not rely on any midpoint computations. Following the normalization framework of Perdomo and Plaza \cite{PerdomoPlaza2012,PerdomoPlaza2014,PerdomoPlaza2013}, we analyze LEAB directly in the space of triangle shapes modulo similarity. In this setting, each triangle shape is represented by a point in a bounded region~$\Sigma$ of the complex upper half-plane, while right triangles correspond to a distinguished hyperbolic geodesic~$\Gamma \subset \Sigma$. Within this framework, LEAB admits an especially transparent geometric description.

Our main result shows that the action of LEAB on shape space is completely determined by a simple ray–geodesic intersection property. Given $z\in\Sigma$ representing a triangle, the normalized left and right children obtained by LEAB are given, up to the normalization symmetry, by the intersections of the geodesic~$\Gamma$ with the rays issued from $0$ and $1$, respectively, and passing through~$z$. This geometric characterization implies that the LEAB dynamical system collapses the entire shape space~$\Sigma$ onto the geodesic~$\Gamma$ in a single refinement step, and that every point of~$\Gamma$ is a fixed point of the refinement.

The paper is organized as follows. In Section~2 we recall the normalized shape space and describe the geometry of the LEAB refinement in this setting, establishing the ray–geodesic intersection property and deriving explicit formulas for the associated refinement maps. The consequences of this geometric collapse, as well as practical considerations such as mesh heterogeneity and conformity, are discussed subsequently.

\section{Geometry of the longest-edge altitude bisection in the shape space}

\begin{definition}[LEAB]
    Given a triangle $T$, the longest-edge altitude bisection subdivides $T$ by dropping the altitude from the vertex opposite its longest edge. The triangle is thus subdivided into two children, which we refer to as the left and right children according to whether the subtriangle shares the left or right endpoint of the longest edge.
\end{definition}

Following \cite{PerdomoPlaza2014}, we study this refinement process as a dynamical system of similarity classes of triangles. Every similarity class of nondegenerate triangles can be represented by a unique point
\[
z \in \Sigma := \{\, z \in \mathbb C : 0 < \Re(z) \le \tfrac12,\ \Im(z)>0,\ |z-1|\le 1 \,\},
\]
obtained after normalizing the triangle so that its longest edge coincides with the unit segment $[0,1]$, the third vertex lies in the upper half–plane, and the shortest edge is attached to~$0$.
The space $\Sigma$ is endowed with the Poincaré metric of the upper half–plane \cite{anderson2005hyperbolic}.

Within this normalized shape space, the right triangles are represented by the hyperbolic geodesic \cite{PerdomoPlaza2014}
\[
\Gamma := \{\, z \in \Sigma : |z-\tfrac12| = \tfrac12 \,\}.
\]

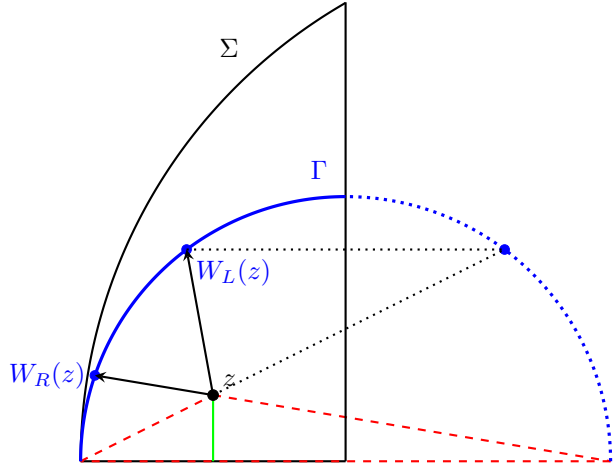
\begin{figure}
    \centering
    \begin{tikzpicture}[scale=7,>=stealth]



\draw[thick] (0,0) -- (0.5,0);


\draw[thick,red,dashed] (0,0) -- (0.25,0.125);

\draw[thick,red,dashed] (0.25,0.125) -- (1,0);
\draw[thick,red,dashed] (1,0) -- (0,0);
\draw[thick,green] (0.25,0) -- (0.25,0.125);
\draw[thick] (0.5,0) -- (0.5,0.867);

\draw[thick,domain=120:180,samples=100]
  plot ({1+cos(\x)}, {sin(\x)});

\node at (0.28,0.78) {$\Sigma$};


\draw[very thick,blue,domain=90:180,samples=80]
  plot ({0.5+0.5*cos(\x)}, {0.5*sin(\x)});

  \draw[very thick,blue,dotted,domain=0:90,samples=80]
  plot ({0.5+0.5*cos(\x)}, {0.5*sin(\x)});

\node[blue] at (0.45,0.55) {$\Gamma$};

\filldraw (0.25,0.125) circle (0.010)
  node[above right] {$z$};

\filldraw[blue] (0.2,0.4) circle (0.009)
  node[below right] {$W_L(z)$};

\filldraw[blue] (0.8,0.4) circle (0.009);

\draw[dotted, thick] (0.2,0.4) -- (0.8,0.4);
\draw[dotted, thick] (0.25,0.125) -- (0.8,0.4);

\filldraw[blue] (0.027027,0.162162) circle (0.009)
  node[left] {$W_R(z)$};

\draw[->,thick] (0.25,0.125) -- (0.2,0.4);
\draw[->,thick] (0.25,0.125) -- (0.027027,0.162162);

\end{tikzpicture}
    \caption{Illustration of the LEAB process. Triangle $z$ (shown in red) is divided into two children triangle by the altitude (green). The normalized left and right children belong to  the geodesic $\Gamma$ in the space of right triangles  (marked blue). The rays issued from $0$ and $1$ are also shown to illustrate how the symmetry around $\Re z$ may be needed to find the correct normalization.
    }
    \label{fig:LEAB-geometry}
\end{figure}

Applying LEAB to a normalized triangle $z$ produces two right triangles. After renormalization, these define two maps
\[
W_L,\, W_R : \Sigma \longrightarrow \Sigma,
\]
corresponding to the left and right children. 

The key geometric observation underlying the LEAB dynamics is that these images can be characterized purely in terms of ray–geodesic intersection.

\begin{theorem}
Let $z \in \Sigma$.
\begin{enumerate}
\item The normalized left child $W_L(z)$ is given by the intersection of the geodesic $\Gamma$ with the ray issued from $0$ and passing through $z$, up to the normalization symmetry imposed by the Perdomo–Plaza convention.
\item The normalized right child $W_R(z)$ is given by the intersection of $\Gamma$ with the ray issued from $1$ and passing through $z$, up to the same symmetry.
\item The formulas for $W_L(z)$ and $W_R(z)$ are given by:
\[
W_L(z)=
\begin{cases}
\dfrac{z+\bar z}{2z}, & \Re(z) \le \Im(z),\\[3ex]
\dfrac{z-\bar z}{2z}, & \Re(z) > \Im(z),
\end{cases}
\qquad
W_R(z)=
\begin{cases}
\dfrac{z+\bar z-2}{2(z-1)}, & 1-\Re(z) \le \Im(z),\\[3ex]
\dfrac{z-\bar z}{2(z-1)}, & 1-\Re(z) > \Im(z).
\end{cases}
\]
\end{enumerate}
\end{theorem}

\begin{proof}
We treat the left child; the argument for the right child is identical.

The (unnormalized) left child of $z$ is the right triangle with vertices
$\{\,0,\,\Re(z),\,z\,\}$.

All left children corresponding to points lying on the ray $\{\,t z : t>0\,\}$ are similar to each other, since they differ only by a homothetic scaling along that ray.

On the other hand, the normalization of a right triangle corresponds to a point of $\Gamma$, and right triangles represented by points on $\Gamma$ are similar to their normalized representatives.
Consequently, the normalized left child of $T(z)$ must lie on $\Gamma$ and belong to the same similarity class as the triangle represented by the ray through $z$.

This uniquely determines $W_L(z)$ as the intersection of $\Gamma$ with the ray from $0$ through $z$, up to the symmetry imposed to ensure that the shortest edge is attached to~$0$ in the normalization.
The same argument applies to the right child, replacing the ray from $0$ by the ray from $1$.

The formulas for $W_L(z)$ and $W_R(z)$ are obtained by straightforward computations. The domain of these maps is illustrated in Figure \ref{fig:LEAB-maps}.
\end{proof}

\begin{remark}
    This geometric description immediately implies that the LEAB dynamics collapses the entire shape space onto the geodesic $\Gamma$ in a single refinement step, as illustrated in Figure~\ref{fig:LEAB-geometry}.
\end{remark}

\begin{figure}[ht]
\centering
\begin{tikzpicture}[scale=7,>=stealth]




\draw[thick] (0,0) -- (0.5,0);

\draw[thick] (0.5,0) -- (0.5,0.867);

\draw[thick,domain=120:180,samples=100]
  plot ({1+cos(\x)}, {sin(\x)});

\node at (0.28,0.78) {$\Sigma$};
\node at (0.28,0.9) {Left map};

\draw[dashed] (0,0) -- (0.5,0.5);
\node[rotate=45] at (0.33,0.37) {$\Re z=\Im z$};


\node[
  fill=white,
  inner sep=2pt,
  align=center
] at (0.3,0.55)
{
$\displaystyle\frac{z+\bar z}{2z}$};

\node[
  fill=white,
  inner sep=2pt,
  align=center
] at (0.35,0.18)
{
$\displaystyle \frac{z-\bar z}{2z}$};

\end{tikzpicture}
\hspace{2cm}
\begin{tikzpicture}[scale=7,>=stealth]



\draw[thick] (0,0) -- (0.5,0);


\draw[thick] (0.5,0) -- (0.5,0.867);

\draw[thick,domain=120:180,samples=100]
  plot ({1+cos(\x)}, {sin(\x)});

\node at (0.28,0.78) {$\Sigma$};
\node at (0.28,0.9) {Right map};

\draw[dashed] (0.3,.7) -- (0.5,0.5);
\node[rotate=-45] at (0.37,0.55) {$1-\Re z=\Im z$};


\node[
  fill=white,
  inner sep=2pt,
  align=center
] at (0.42,0.7)
{\small
$\frac{z+\bar z-2}{2(z-1)}$};

\node[
  fill=white,
  inner sep=2pt,
  align=center
] at (0.25,0.25)
{
$\displaystyle \frac{z-\bar z}{2(z-1)}$};

\end{tikzpicture}
\caption{Piecewise definition of the left and right LEAB maps $W_L$ and $W_R$ on the normalized shape space $\Sigma$.}
\label{fig:LEAB-maps}
\end{figure}
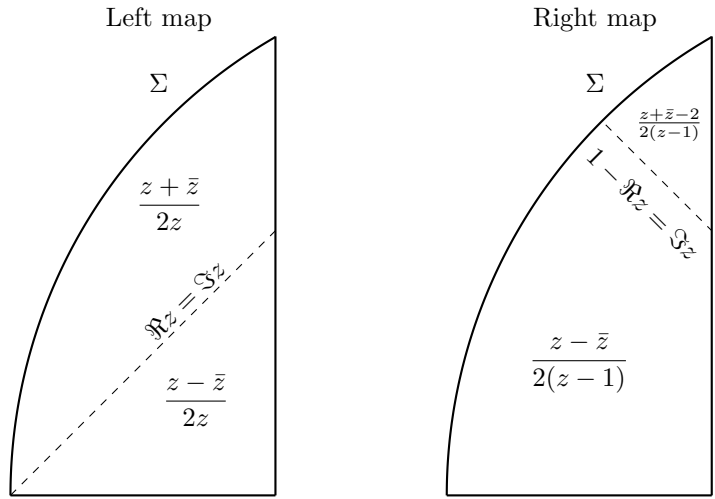

\begin{remark}
    In contrast with what happens for bisection, trisection, and $n$-section refinements, $W_L$ and $W_R$ are not piecewise Möbius transformations, see \cite{PerdomoPlaza2012,PerdomoPlaza2013,PerdomoPlaza2014}. 
\end{remark}

\section{Discussion}
The geometric description of the longest-edge altitude bisection (LEAB) developed in Section~2 shows that the refinement induces an extreme form of rigidity at the level of triangle shapes: in a single step, the entire normalized shape space~$\Sigma$ collapses onto the geodesic~$\Gamma$ of right triangles, and all subsequent refinements take place within this one-parameter family of fixed points. This behavior stands in sharp contrast with LEB, whose induced dynamics involves several invariant similarity classes and nontrivial refinement orbits, as described in the hyperbolic shape-space framework of Perdomo and Plaza~\cite{PerdomoPlaza2012,PerdomoPlaza2014}.

\begin{figure}[h]
    \centering
    \begin{tikzpicture}[scale=3]
    \coordinate (A) at (0,0);
\coordinate (B) at (1,0);
\coordinate (C) at ({0.5},{sqrt(3)/2});
\draw[thick] (0,0) -- (1,0) -- ({0.5},{sqrt(3)/2}) -- cycle;
\draw[dashed] ({0.5},{sqrt(3)/2}) -- ({0.5},0);
\node[below] at ($(A)!0.5!(B)$) {$l$};

\node[left]  at ($(A)!0.5!(C)$) {$l$};

\node[right] at ($(B)!0.5!(C)$) {$l$};
\end{tikzpicture}
\hspace{2cm}
\begin{tikzpicture}[scale=4]

\coordinate (A) at (0,0);
\coordinate (B) at (1,0);
\coordinate (C) at (0.3,{sqrt(0.21)});

\coordinate (H) at ($(A)!(C)!(B)$);

\draw[thick] (A) -- (B) -- (C) -- cycle;

\draw[thick,dashed] (C) -- (H);

\draw (A) -- (C)
  node[midway,sloped,above] {$l\sin\alpha$};

\draw (B) -- (C)
  node[midway,sloped,above] {$l\cos\alpha$};

\draw (A) -- (B)
  node[midway,below] {$l$};



\end{tikzpicture}
    \caption{Illustration of the LEAB process on an equilateral and a right triangle.}
    \label{fig:placeholder}
\end{figure}
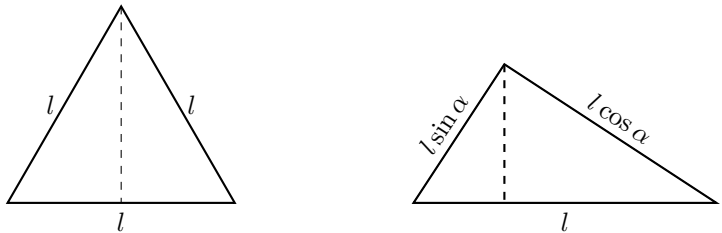

From a metric point of view, LEAB exhibits a two-phase behavior that directly reflects this geometric collapse. In the first refinement step, no uniform diameter contraction occurs in general. 
Similar regularization effects without immediate contraction are also observed in  other longest-edge based schemes and have been noted in early analyses of refinement processes~\cite{RosenbergStenger1975,Stynes1980}.

After this initial step during which no diameter reduction may occur, see Figure \ref{fig:placeholder}, all triangles are right triangles whose longest edge is always the hypotenuse. Theorem~\ref{thm:two-sided} provides two-sided bounds for the diameter reduction when applying LEAB. Such explicit upper and lower bounds on element sizes are consistent with classical angle-based mesh regularity theory~\cite{BabuskaAziz1976} and with similar estimates arising in the context of graded and anisotropic meshes~\cite{Apel1999,BankSmith1997}.

\begin{theorem}[Two-sided diameter bounds for LEAB]\label{thm:two-sided}
Let $\{\mathcal T_k\}_{k\ge0}$ be the triangulations generated by successive global
applications of the longest-edge altitude bisection (LEAB), starting from an initial
triangulation $\mathcal T_0$ with minimal interior angle $\alpha_0>0$.
Let
\[
c=\min\{\operatorname{diam}(T):T\in\mathcal T_1\},\qquad
C=\max\{\operatorname{diam}(T):T\in\mathcal T_1\},
\]
and let $\alpha_1\ge\tfrac12\alpha_0$ denote the minimal interior angle in $\mathcal T_1$.
Then for every $k\ge1$ and every $T\in\mathcal T_k$,
\[
c\,(\sin\alpha_1)^{k-1}
\;\le\;
\operatorname{diam}(T)
\;\le\;
C\,(\cos\alpha_1)^{k-1}.
\]
In particular, $\{\mathcal T_k\}$ is uniformly refining, while the diameters of elements
within a fixed level $\mathcal T_k$ may differ by arbitrarily large factors when
$\alpha_0$ is small.
\end{theorem}

\begin{proof}
The first refinement step does not produce a uniform diameter contraction in general.
For instance, as shown in Figure~\ref{fig:placeholder} (left), an equilateral triangle is subdivided into two
congruent right triangles whose diameters coincide with that of the parent, while the
smallest interior angle is reduced by a factor two. This step should therefore be
viewed as a geometric regularization.

After this initial step, all triangles in $\mathcal T_1$ are right triangles whose
longest edge is the hypotenuse. As illustrated in Figure~\ref{fig:placeholder} (right), dropping the altitude
to the hypotenuse subdivides such a triangle into two right triangles similar to the
parent. If $\alpha_1$ denotes the smallest interior angle, the hypotenuse lengths of the
children are reduced by the factors $\cos\alpha_1$ and $\sin\alpha_1$. Repeating this
argument along the refinement process yields the stated two-sided bounds.
\end{proof}

\begin{remark}
The bounds in Theorem~3.1 also reveal an intrinsic limitation of LEAB. When the minimal
angle $\alpha_0$ is small, the contraction factor $\cos\alpha_1$ is close to one, leading
to slow decay of the largest elements, while $\sin\alpha_1$ may be very small, causing
rapid decay of the smallest ones. As a result, LEAB can generate highly heterogeneous
triangulations, even though interior angles remain uniformly bounded from below after the
first refinement step. This contrasts with longest-edge bisection, where refinement is
distributed more evenly among similarity classes and the shape of elements seems to be improving; see, for instance,
\cite{Rivara1984,PerdomoPlaza2014,kalmanovich2026stability}.
\end{remark}

Another important practical aspect is that LEAB does not, in general, produce conforming triangulations. The position of the altitude foot on the longest edge depends on the individual geometry of each triangle and typically does not coincide with subdivision points introduced in neighboring elements. Consequently, local refinement by LEAB may lead to hanging nodes unless additional closure or refinement propagation rules are enforced. This issue is well known for geometry-driven and longest-edge based refinement schemes and has been extensively studied in the context of adaptive mesh refinement; see, for example, Rivara~\cite{Rivara1984}, Mitchell~\cite{Mitchell1991}, and Traxler~\cite{Traxler1997}.

While the lack of conformity and the potential heterogeneity of element sizes may limit the direct applicability of LEAB as a refinement strategy for standard finite element methods -- where conformity and balanced grading are often essential -- these features do not constitute intrinsic drawbacks of the refinement itself. In other contexts, such as adaptive subdivision methods for nonlinear root finding, continuation algorithms, or geometric approximation procedures, the rigidity and transparency of the LEAB dynamics may be advantageous. In such settings, refinement strategies based on explicit geometric control rather than conformity have been successfully employed; see, for example, Allgower and Georg~\cite{AllgowerGeorg1990}.

From this perspective, LEAB is best understood not as a competitor to classical longest-edge bisection, but as a conceptual limit case. Its finite-time collapse in shape space illustrates how hyperbolic shape-space techniques can be used to understand not only complex refinement dynamics, but also refinement schemes whose induced dynamics degenerate in a highly structured manner. This makes LEAB a useful reference example for exploring the interplay between geometric refinement rules and the global dynamics they induce on spaces of triangle shapes.

\bibliographystyle{abbrv}
\bibliography{references}

@article{BabuskaAziz1976,
  author  = {I. Babu{\v{s}}ka and A. K. Aziz},
  title   = {On the angle condition in the finite element method},
  journal = {SIAM Journal on Numerical Analysis},
  volume  = {13},
  number  = {2},
  year    = {1976},
  pages   = {214--226}
}

@article{RosenbergStenger1975,
  author  = {I. G. Rosenberg and F. Stenger},
  title   = {A lower bound on the angles of triangles constructed by bisecting the longest side},
  journal = {Mathematics of Computation},
  volume  = {29},
  year    = {1975},
  pages   = {390--395}
}

@article{Rivara1984,
  author  = {M. C. Rivara},
  title   = {Mesh refinement processes based on the generalized bisection of simplices},
  journal = {SIAM Journal on Numerical Analysis},
  volume  = {21},
  year    = {1984},
  pages   = {604--613}
}

@article{Stynes1980,
  author  = {M. Stynes},
  title   = {On faster convergence of the bisection method for all triangles},
  journal = {Mathematics of Computation},
  volume  = {35},
  year    = {1980},
  pages   = {1195--1201}
}

@article{PerdomoPlaza2014,
  author  = {F. Perdomo and A. Plaza},
  title   = {Properties of triangulations obtained by the longest-edge bisection},
  journal = {Central European Journal of Mathematics},
  volume  = {12},
  number  = {12},
  year    = {2014},
  pages   = {1796--1810}
}

@article{PerdomoPlaza2012,
  author  = {F. Perdomo and A. Plaza},
  title   = {A new proof of the degeneracy property of the longest-edge $n$-section refinement scheme for triangular meshes},
  journal = {Applied Mathematics and Computation},
  volume  = {219},
  year    = {2012},
  pages   = {2342--2344}
}

@article{PerdomoPlaza2013,
  author  = {F. Perdomo and A. Plaza},
  title   = {Proving the non-degeneracy of the longest-edge trisection by a space of triangular shapes with hyperbolic metric},
  journal = {Applied Mathematics and Computation},
  volume  = {221},
  year    = {2013},
  pages   = {424--432}
}

@article{Mitchell1991,
  author  = {Mitchell, W. F.},
  title   = {Adaptive refinement for arbitrary finite-element spaces with hierarchical bases},
  journal = {Journal of Computational and Applied Mathematics},
  volume  = {36},
  number  = {1},
  pages   = {65--78},
  year    = {1991},
  doi     = {10.1016/0377-0427(91)90228-R}
}

@article{Traxler1997,
  author  = {Traxler, C. T.},
  title   = {An algorithm for adaptive mesh refinement in two dimensions},
  journal = {Computing},
  volume  = {59},
  number  = {2},
  pages   = {115--137},
  year    = {1997},
  doi     = {10.1007/BF02684465}
}

@book{Apel1999,
  author    = {Apel, Thomas},
  title     = {Anisotropic Finite Elements: Local Estimates and Applications},
  publisher = {B. G. Teubner},
  address   = {Stuttgart},
  year      = {1999},
  isbn      = {978-3519026129}
}

@article{BankSmith1997,
  author  = {Bank, Randolph E. and Smith, R. K.},
  title   = {A posteriori error estimates based on hierarchical bases},
  journal = {SIAM Journal on Numerical Analysis},
  volume  = {30},
  number  = {4},
  pages   = {921--935},
  year    = {1993},
  doi     = {10.1137/0730047}
}

@book{AllgowerGeorg1990,
  author    = {Allgower, Eugene L. and Georg, Kurt},
  title     = {Numerical Continuation Methods: An Introduction},
  publisher = {Springer},
  address   = {Berlin},
  year      = {1990},
  series    = {Springer Series in Computational Mathematics},
  volume    = {13},
  doi       = {10.1007/978-3-642-61257-2}
}

@book{anderson2005hyperbolic,
  title={Hyperbolic geometry},
  author={Anderson, James},
  year={2005},
  publisher={Springer Science \& Business Media}
}

@article{kalmanovich2026stability,
  title={On the stability, complexity, and distribution of similarity classes of the longest edge bisection process for triangles},
  author={Kalmanovich, Daniel and Solomon, Yaar},
  journal={arXiv preprint arXiv:2601.13663},
  year={2026}
}
\end{document}